\documentclass{article}

\usepackage{setspace}
\usepackage{latexsym}
\usepackage{amsfonts}
\usepackage{amscd}
\usepackage{amsmath, amsthm, amssymb}
\usepackage{bbm}
\usepackage[all]{xy}
\usepackage{cite}

\author{Thomas John Baird}
\title{ Moduli spaces of flat SU(2)-bundles over nonorientable surfaces}

\newtheorem{thm}{Theorem}[section]
\newtheorem{cor}[thm]{Corollary}
\newtheorem{lem}[thm]{Lemma}
\newtheorem{prop}[thm]{Proposition}

\theoremstyle{definition}

\newtheorem{define}{Definition}
\newtheorem{rmk}{Remark}

\newcommand{\lie}[1]{\mathfrak{#1}}

\newcommand{\im}{\mathrm{im}}
\newcommand{\id}{\mathbbmss{1}}
\newcommand{\Z}{\mathbb{Z}}
\newcommand{\R}{\mathbb{R}}
\newcommand{\C}{\mathbb{C}}

\newcommand{\Q}{\mathbb{Q}}

\newcommand{\co}{:}

\newcommand{\ti}[1]{\tilde{#1}}

\newcommand{\gau}{\mathcal{G}}
\begin{document}


\maketitle

\begin{abstract}
We study the topology of the moduli space of flat $SU(2)$-bundles over a \emph{non}orientable surface $\Sigma$.  This moduli space may be identified with the space of homomorphisms $Hom(\pi_1(\Sigma),SU(2))$ modulo conjugation by $SU(2)$. In particular, we compute the (rational) equivariant cohomology ring of $Hom(\pi_1(\Sigma),SU(2))$ and use this to compute the ordinary cohomology groups of the quotient $Hom(\pi_1(\Sigma),SU(2))/SU(2)$. A key property is that the conjugation action is equivariantly formal.
\end{abstract}








\section{Introduction}\label{resout}

Let $G$ be a compact, simply connected Lie group and let $\Sigma$ be a compact surface without boundary. Denote by $\mathcal{A}$ the space of connections on the trivial $G$-bundle over $\Sigma$. The \textbf{moduli stack of flat $G$-bundles over $\Sigma$} is the topological quotient stack $[\mathcal{A}_{flat} / \mathcal{G}]$, where $\mathcal{A}_{flat} \subset \mathcal{A}$ is the subspace of flat connections and $\mathcal{G} \cong C^{\infty}(\Sigma, G)$ is the gauge group. We call the orbit space $\mathcal{A}_{flat} / \mathcal{G}$ the \textbf{coarse moduli space of flat $G$-bundles over $\Sigma$}. In this paper we compute the cohomology of $[\mathcal{A}_{flat} / \mathcal{G}]$ and $\mathcal{A}_{flat} / \mathcal{G}$ with rational coefficients in the case $G=SU(2)$.

For orientable $\Sigma$, the most successful approach to studying the topology of $[\mathcal{A}_{flat} / \mathcal{G}]$ was initiated by Atiyah-Bott \cite{ab2}. The Yang-Mills functional, $YM:\mathcal{A} \rightarrow \R$, achieves its minimum on $\mathcal{A}_{flat}$, and by using Morse theory they are able to prove that the induced map:

\begin{equation}\label{grab}
H^*_{\mathcal{G}}(\mathcal{A}) \rightarrow H^*_{\mathcal{G}}(\mathcal{A}_{flat})
\end{equation}
is surjective (unless otherwise stated, coefficients are assumed to be rational). Since $\mathcal{A}$ is a contractible space, $H^*_{\mathcal{G}}(\mathcal{A}) = H^*(B \mathcal{G})$ which is a freely generated (super)commutative ring providing a convenient set of generators for $H_{\mathcal{G}}(\mathcal{A}_{flat})$. Most subsequent work has revolved around trying to understand the relations between generators or, equivalently, the kernel of (\ref{grab}).

Yang-Mills Morse theory has also been used to study the topology of $[\mathcal{A}_{flat}/\mathcal{G}]$ for nonorientable $\Sigma$. Morse inequalities and the number of connected components were obtained by Ho-Liu \cite{hl} \cite{hl2} and the stable homotopy under the rank of $G$ was determined by Ramras \cite{ram}. Nevertheless, the Morse theory approach is less effective for nonorientable $\Sigma$. In \S \ref{morgau} we prove

\begin{thm}\label{ale}
For nonorientable $\Sigma$, the map
\begin{equation}
H_{\gau}(\mathcal{A}) \rightarrow H_{\gau}(\mathcal{A}_{flat})
\end{equation}
is injective. 
\end{thm}

Thus for nonorientable $\Sigma$ the Atiyah-Bott strategy does not work. Fortunately, the stack $[\mathcal{A}_{flat}/\gau]$ has other nice properties that compensate for the failure of Yang-Mills Morse theory.

These properties are best understood using a finite dimensional model for the moduli stack. Let $\pi = \pi_1(\Sigma)$ denote the fundamental group of the surface and let $Hom(\pi, G)$ denote the set of group homomorphisms from $\pi$ to $G$. This set carries the structure of a finite dimensional real algebraic variety and is acted upon by $G$ via conjugation. Holonomy describes a morphism of spaces $hol: \mathcal{A}_{flat} \rightarrow Hom(\pi, G)$ which determines an isomorphism of topological stacks:

$$ [\mathcal{A}_{flat}/ \mathcal{G}] \cong [Hom(\pi, G)/G].$$
In particular, we have an isomorphism of equivariant cohomology rings $H_{\mathcal{G}}(\mathcal{A}_{flat}) \cong H_G(Hom(\pi, G))$.

\subsection{Summary of results}

From this point onwards, set $G= SU(2)$ and $T \subset G$ the maximal torus of diagonal matrices, so $T \cong U(1)$.

Recall that every compact nonorientable surface without boundary is isomorphic to a connected sum of real projective planes: $$ \Sigma_n := (\R P^2)^{\#n+1}$$
The fundamental group $\pi_1(\Sigma_n)$ is  has presentation:

$$ \pi_1(\Sigma_n) = \{ a_0,..., a_n| \prod_{i=0}^n a_i^2 = \id\}.$$
A homomorphism in $Hom(\pi_1(\Sigma_n), G)$ is determined by where it sends the generators, so we may identify $$Hom(\pi_1(\Sigma_n),G) \cong X_n(\id),$$ where

$$ X_n(\id) := \{ (g_0,...,g_n) \in G^{n+1} |  \prod_{i=0}^n g_i^2 = \id\}.$$
Indeed, for $\mathcal{C} \subset G$ a conjugacy class, we will consider more generally varieties 

\begin{equation}\label{inchty}
X_n(\mathcal{C}) := \{(g_0,...,g_n) \in G^{n+1} |  \prod_{i=0}^n g_i^2 \in \mathcal{C}\}.
\end{equation}
The group $G$ acts on $X_n(\mathcal{C})$ via conjugation.   
The quotient stack $[X_n(\mathcal{C})/G]$ also has a gauge theoretic interpretation as the moduli space of flat connections over a punctured $\Sigma_n$ with restricted holonomy around the puncture. 

Except for the central elements $\pm \id$, all conjugacy classes are isomorphic as $G$-spaces to the homogeneous space $G/T \cong S^2$. We prove in \S \ref{Regular and singular} that as $G$-spaces $$ X_n(\mathcal{C}) \cong G/T \times X_n( (-\id)^n)$$ under the diagonal action, and consequently we focus our attention on the two special cases $X_n(\pm \id)$.

One of the main results of this paper is:
\begin{thm}\label{formalityresult}
Under the restricted action by $T \subset G$, $X_n(\pm \id)$ is equivariantly formal for all $n \in \Z_{\geq 0} $. In particular, as a graded vector space $H_T^*(X_n(\pm \id)) \cong H^*(X_n(\pm \id)) \otimes H^*(BT)$, and we get an injection:

\begin{equation}\label{firstup}
 i^*: H_T^*(X_n(\pm \id)) \hookrightarrow H_T^*(X_n(\pm \id)^T)
\end{equation}
where $i: X_n(\pm \id)^T \hookrightarrow X_n( \pm \id)$ denotes the inclusion of the $T$-fixed points.
\end{thm}
In fact, $T$-equivariant formality is equivalent to $G$-equivariant formality for a $G$-space, so it is also true that

$$ H_G^*(X_n(\pm \id)) \cong H^*(X_n(\pm \id)) \otimes H^*(BG)$$
as an $H^*(BG)$-module and we we have a natural isomorphism of rings $H_G^*(X_n(\pm \id)) \cong H_T^*(X_n(\pm \id))^W$, where $W\cong \Z_2$ is the Weyl group. We choose to work with the restricted $T$-action in order to exploit the localization theorem of Borel, which relates a $T$-space to its fixed point set.

Theorem \ref{formalityresult} is the crucial result that makes the remaining computations tractable. It is somewhat surprising, because the analogous statement is not true for orientable surfaces, where in many cases the action is actually free.

We prove Theorem \ref{formalityresult} by using an induction argument on $n$ to compute $H^*(X_n(\pm \id))$, and then using the criterion that a (reasonable) compact $T$-space $X$ is equivariantly formal if and only if $\dim H^*(X) = \dim H^*(X^T)$ (see Proposition \ref{forcrit}).

The variety $X_n(\pm\id)$ possesses an action by the centre $ Z(SU(2)) \cong \Z_2$ by multiplying (say) the 0th factor of an (n+1)-tuple. This makes sense because elements of $Z(SU(2))$ square to $\id$ and so leave the defining relation (\ref{inchty}) invariant. Our study of the localization map uncovers the following interesting cohomological identity, which is an isomorphism of graded algebras:

\begin{equation}\label{reduc}
H_T^*( X_{n+1}(\epsilon)) \cong H_T^*(X_n(-\epsilon) \times_{\Z_2} X_1(-\id)) \text{,    $\epsilon = \pm \id$}
\end{equation}
where the right side is the orbit space by the aforementioned $\Z_2$-action acting diagonally. Notice that applying (\ref{reduc}) inductively we may express the cohomology of $X_{n}(\epsilon)$ in terms of the $n=1$ case (i.e. the Klein bottle). Equation (\ref{reduc}) does not come from a homeomorphism of spaces, because the cohomology isomorphism does not hold integrally, and we have so far been unsuccessful in producing (\ref{reduc}) via a map of spaces. In future work we intend to produce evidence that (\ref{reduc}) holds when $SU(2)$ is replaced by an arbitrary compact connected Lie group, under suitable genericity conditions.

Another interesting result is the presence of a bigrading on $H_T^*(X_n(\pm \id))$.
Recall that for the $T$-fixed point set, we have a canonical isomorphism of rings $$H^*_T( X_n(\pm \id)^T) \cong H^*(X_n(\pm \id)^T)\otimes H^*(BT)$$ so $H_T^*( X_n(\pm \id)^T)$ inherits a bigrading in a natural way. In \S \ref{local}, we compute the image of the localization map (\ref{firstup}) and discover that the localization map is compatible with this bigrading, i.e.

\begin{cor}\label{bigradient}
For $n \geq 0$ and $\epsilon = \pm \id$, the equivariant cohomology of the representation variety $H_T^*(X_n(\epsilon))$ possesses a bigrading for which $$i^*: H_T^*(X_n(\epsilon)) \rightarrow H^*(X_n(\epsilon)^T) \otimes H^*(BT)$$ is a morphism of bigraded algebras. Furthermore, the bigrading descends to one on ordinary cohomology $H^*(X_n(\epsilon))$.
\end{cor}
Under this bigrading, (\ref{reduc}) holds as bigraded rings. In a future paper we intend to show that the second grading can also be obtained by representing $H(X_n(\epsilon))$ as a module for the $Ext$-algebra of a contructible sheaf complex over $SU(2)$.

In \S \ref{chapter7} we compute the ordinary, rational cohomology of the orbit spaces $X_n(\mathcal{C})/G$ and the compactly supported cohomology of its smooth locus. The technique we use is borrowed from Cappell-Lee-Miller \cite{clm}, and can be applied to any reasonable equivariantly formal $SU(2)$-space.  We are able to almost, but not completely, determine the cup product structure, which we reduce to a Hochschild extension problem.




The outline of the paper is as follows. In \S \ref{locusoffixedpoints} we study the $T$-fixed point sets of $X_n(\pm \id)$, emphasizing properties that will remain true for the full variety. In \S \ref{Cohomological orbit spaces} we introduce the notion of a cohomological Galois cover and of a cohomological orbit space. The heart of the paper is \S \ref{chapter5}. In \S \ref{almostfinshed} we demonstrate that $X_n(\pm \id)$ is a 2-fold cohomological Galois cover over $G^n$. In \S \ref{bettys} we compute the Betti numbers of $X_n(\pm \id)$, proving Theorem \ref{formalityresult} and thus determining the Poincar\'e series of $H^*_{\mathcal{G}}(\mathcal{A}_{flat})$. In \S \ref{rockthecasbah} we consider torsion in the cohomology. In \S \ref{morgau} we explore the implications for the Morse theory of the Yang-Mills functional. In \S \ref{relationships1} and \S \ref{relationships2} we explain the factorization formula (\ref{reduc}), postponing the proof until \S \ref{chapter6}. This in particular determines the ring structure of $H^*_{\mathcal{G}}(\mathcal{A}_{flat}) $ In \S \ref{chapter7} we compute the ordinary, rational cohomology of the orbit spaces $X_n(\mathcal{C})/G$ and in particular compute $H^*(\mathcal{A}_{flat}/\gau)$.

\emph{Notational conventions}: Unless otherwise stated, cohomology is singular with $\Q$ coefficients. Unless there is risk of confusion, we omit the superscript $*$ denoting cohomology, i.e. $H(.) = H^*(.)$. We use $G$ to denote $SU(2)$ and $T$ to denote the maximal torus of diagonal matrices in $SU(2)$. 

\emph{Funding}: This research was supported in part by an NSERC doctoral scholarship.

\emph{Acknowledgments}: This paper was motivated by questions of Raoul Bott, Kenji Fukaya and Lisa
Jeffrey about the Betti numbers of moduli spaces of flat connections over nonorientable
surfaces. It is based on the author's doctoral thesis conducted under the supervision of Lisa Jeffrey and Paul Selick and I thank them for all their guidance and support. Thanks also to Fred Cohen, Tom Goodwillie, Ian Hambleton, Nan-Kuo Ho, Geoff Lynch, Eckhard Meinrenken, Sue Tolman and Masrour Zoghi for helpful discussions and to Melissa Liu, Johan Martens and Dan Ramras for commenting on earlier drafts.

\section{Locus of T-fixed points}\label{locusoffixedpoints}

The fixed point loci for the maximal torus action on $X_n(\pm \id)$ will play an important role in what follows so we take time to describe it in some detail.

Recall from \S \ref{resout} that $X_n(+\id) \cong Hom(\pi_1(\Sigma_n), G)$. A maximal torus $T \subset G = SU(2)$ is maximal abelian, so it follows that a homomorphism $\phi$ is fixed by $T$ if and only if $\im( \phi) \subset T$, thus  
$$ X_n(+\id)^T \cong Hom(\pi_1(\Sigma_n), G)^T = Hom(\pi_1(\Sigma_n),T).$$
Any homomorphism from $\pi_1(\Sigma_n)$ to an abelian group must factor throught the abelianization $ \pi_1(\Sigma_n) / [ \pi_1(\Sigma_n), \pi_1(\Sigma_n)] \cong H_1( \Sigma_n, \Z) \cong \Z^n \oplus \Z_2$ and we obtain

$$ X_n(+\id)^T = Hom( H_1( \Sigma_n; \Z), T) \cong T^n \times \Z_2.$$
Since $X_n(+\id)^T$ is a group, it acts on itself by left multiplication. The identity component, $X_n(+\id)^{T}_0$ acts trivially on cohomology, so we obtain an action by the quotient $$ X_n(+\id)^T / X_n(+\id)^{T}_0 \cong \Z_2$$ on the cohomology ring $H^*(X_n(+\id)^T)$.

It is an easy exercise to show that $X_n(-\id)^T$ is a $X_n(+\id)^T$-torsor (i.e. is acted on freely and transitively by $X_n(+\id)^T$) so it is also diffeomorphic to $T^n \times \Z_2$ and inherits a $\Z_2$-action on cohomology.

\begin{lem}\label{twotoris}
For $\epsilon = \pm \id$, the fixed point loci $X_n(\epsilon)^T$ is diffeomorphic to $ T^n \times \Z_2$ and comes equipped with a component switching involution determined up to isotopy, inducing a $\Z_2$-action on cohomology.
\end{lem}

It will be useful to describe this involution more explicitly. For $\epsilon = \pm \id$, the set $X_n(\epsilon)^T \subset G^{n+1}$ satisfies:
\begin{equation}\label{fixgh}
X_n(\epsilon)^T = \{ (t_0,..., t_n) \in T^{n+1} | \prod_{i=0}^n t_i^2 = \epsilon \}.
\end{equation}

\begin{lem}\label{deckd}
For $\epsilon = \pm \id$, the projection map  $\rho_T: X_n(\epsilon)^T \rightarrow T^n$ sending $(t_0,...,t_n)$ to $(t_1,...,t_n)$ is a trivial 2-fold covering map. The deck tranformation group $\Z_2$ acts by multiplying the zeroth factor by $-\id$. 
\end{lem}

\begin{proof}
Straightforward.
\end{proof}

Of course there is nothing special about the $0$th coordinate.  Multiplying any other factor by $-\id$ defines an isotopic transformation and induces the same $\Z_2$-action on cohomology.

Decomposing $H(X_n(\epsilon)^T)$ in eigenspaces for the action detemines a $\Z_2$-grading, and it follows easily from Lemma \ref{deckd} that we obtain an isomorphism of $\Z_2\oplus \Z$-graded rings:

\begin{equation*}
H(X_n(\epsilon)^T) \cong \Q \Z_2 \otimes H(T^n)
\end{equation*}
where $\Q \Z_2$ is the group ring of $\Z_2$.
If we use subscripts $\pm $ to distinguish the $\pm 1$ eigenspaces of the action, we obtain
\begin{equation}\label{nearlydoneyeah2}
H(X_n(\epsilon)^{T})_+ \cong H(X_n(\epsilon)^{T})_- \cong H(T^n)
\end{equation}
as $\Z$-graded vector spaces. We will use (\ref{nearlydoneyeah2}) later to describe the localization map in equivariant cohomology. 

Let $X_m(\epsilon_1)^T \times_{\Z_2} X_n(\epsilon_2)^T$ denote the quotient of  $X_m(\epsilon_1)^T \times X_n(\epsilon_2)^T$ by the diagonal $\Z_2$ action.

\begin{lem}\label{popeurban}
Let $\epsilon_i \in \{ \pm \id \}$ for $i = 1,2$.
We have a diffeomorphism 
\begin{equation}\label{montana}
  X_m(\epsilon_1)^T \times_{\Z_2} X_n(\epsilon_2)^T \cong X_{m+n}(\epsilon_1 \epsilon_2)^T.
\end{equation}
\end{lem}

\begin{proof}
It follows easily from Lemma \ref{deckd} that both sides of (\ref{montana}) are diffeomorphic to $T^{m+n} \times \Z_2$. Explicitly, the map $$\phi: X_m(\epsilon_1)^T \times_{\Z_2} X_n(\epsilon_2)^T \rightarrow X_{m+n}(\epsilon_1 \epsilon_2)^T$$
defined by $\phi( (\pm g_0, g_1,..., g_m) \times (\pm h_0, ..., h_n)) = (g_0h_0, g_1,...,g_m, h_1,...,h_n)$ produces a diffeomorphism.
\end{proof}

The $\Z_2$-action on $X_n(\epsilon)^T$ extends to the full space $X_n(\epsilon)$ in an obvious way. The projection map $ X_n(\epsilon) \rightarrow G^n$ is preserved by the $\Z_2$-action but is no longer a covering map. Nevertheless, in \S \ref{almostfinshed} we will show that it behaves like a Galois cover at the level of cohomology.

\section{Cohomological orbit spaces}\label{Cohomological orbit spaces}

In this section we introduce the notion of a cohomological orbit space. This is a variation on the notion of a cohomological principal bundle, from \cite{b1} and \cite{b2} .

\begin{define}\label{cohofib}
Given topological spaces $X$ and $Y$, we say that a continuous map $f\co X \rightarrow Y$ is a \textbf{cohomological orbit space} for the cohomology theory $H$ if there exists a group $\Gamma$ and a $\Gamma$ space $\ti{X}$ fitting into a commutative diagram

\begin{equation}\begin{CD}
\xymatrix{ \ti{X} \ar[r] \ar[d]^{\pi} & X \ar[d]^f \\
           \ti{X}/\Gamma \ar[r] & Y }
\end{CD}\end{equation}
such that the horizontal arrows induce isomophisms in H-cohomology.
\end{define}
In the special case that $\ti{X} \rightarrow \ti{X}/\Gamma$ is a finite Galois cover (i.e. a principal bundle with finite structure group), we say that $ X \rightarrow Y$ is a \textbf{finite cohomological Galois cover} for $H$.

\begin{define}\label{princ}
Let $ f\co X \rightarrow Y $ be a continuous map between topological
spaces $X$ and $Y$ where $X$ is a paracompact Hausdorff space, and let $ \Gamma $ be a finite group
acting on the right of $X$.
We say $(f\co X  \rightarrow Y, \Gamma )$ is a \textbf{strong cohomological
orbit space} for the cohomology theory H if:
\\
i) $f$ is a closed surjection
\\
ii) $f$ descends through the quotient to a map $h$,
\\
\begin{equation*}\begin{CD}
\xymatrix{ X \ar[d]^{\pi} \ar[r]^{id_X}& X \ar[d]^f \\
       X/\Gamma \ar[r]^h & Y}
\end{CD}\end{equation*}
\\
iii) $H(h^{-1}(y)) \cong H(pt)$ for all $y \in Y$
\end{define}

The terminology above is somewhat modified from \cite{b1} because new examples were found in \cite{b2} and in the present paper which require a broader definition. The next proposition is a rewording and slight generalization of Proposition 2.3 from \cite{b1}, and is proven using the same method. 

Let $H(X;F)$ denote sheaf cohomology of the constant sheaf $F_X$, where $F$ is a field (in all applications we have in mind, sheaf cohomology is isomorphic to singular cohomology).
\begin{prop}\label{strongweak}
For sheaf cohomology $H(.;F)$, any strong cohomological orbit space is also a cohomological orbit space.
\end{prop}

Our main interest in cohomological orbit spaces derives from the following standard fact about finite group actions (Bredon \cite{br} II 19.2).

\begin{thm}\label{cov} Let $X$ be a topological space, let $\Gamma$ be a finite group acting on $X$ and let $\pi\co X \rightarrow X/\Gamma$ denote the quotient map onto the orbit space $X/\Gamma$. The induced map $\pi^*$ restricts to an isomorphism

\begin{equation*}
\pi^*\co H(X/\Gamma; \Q) \rightarrow H(X;\Q)^{\Gamma}
\end{equation*}
where $H(X; \Q)^{\Gamma}$ denotes the ring of $\Gamma$ invariants.
\end{thm}


As an immediate consequence we deduce:

\begin{cor}\label{invt}
Let $f\co X \rightarrow Y$ be a cohomological orbit space for $H(.;\Q)$ with finite structure group $\Gamma$. Then $f^*\co H(Y;\Q) \cong H(X;\Q)^{\Gamma}$.
\end{cor}

We remark that in Corollary $\ref{invt}$, $\Q$ could be replaced by a field $F$ with characteristic relatively prime to the order of $\Gamma$.

\section{Cohomology of $X_n(\epsilon)$}\label{chapter5}

In this section, we will compute the Poincar\'e polynomial for the spaces $X_{n}(\pm \id)$. We remind the reader that $G$ denotes the Lie group $SU(2)$, $\lie{g} = \lie{su}(2)$ its Lie algebra and $T \subset G$ the maximal torus of diagonal matrices in $SU(2)$.

\subsection{Regular and singular values of $\Box_n$}\label{Regular and singular}

Define the map  $\Box_n : G^{n+1} \rightarrow G$ by $$\Box_n((g_0,...,g_n)) = \prod_{i=0}^n g_i^2.$$ Clearly we have $ X_n(\epsilon) = \Box_n^{-1}( \epsilon)$ for $\epsilon = \pm \id$.

\begin{prop}\label{singu}
The only singular value of $\Box_{n}: SU(2)^{n+1} \rightarrow SU(2)$ is $(-\id)^{n+1}$.
\end{prop}

\begin{proof}
Using right invariant vector
fields, identify
\begin{equation}\label{trivi}
TG^{n+1} \cong G^{n+1} \times \lie{g}^{n+1}.
\end{equation}
Under this identification, the tangent map $T\Box_n$ is:

\begin{equation*}
T\Box_{n,(g_0,...,g_n)}(\xi_0,...,\xi_n) = \sum_{k=0}^{n}
Ad_{(g_0^2...g_{k-1}^2)}( \xi_k + Ad_{g_k}(\xi_k))
\end{equation*}
where $ g_i \in G$ and $\xi_i \in \lie{g}$.
In particular $(g_0,...,g_n)$ is a regular point for $\Box_n$ as long
as $Id_{\lie{g}} + Ad_{g_i}$ is nonsingular for at least one $i \in \{0,...,n\}$. The adjoint action of G on $\lie{g}$ acts through $SO(\lie{g})\cong SO(3)$, so $Id_{\lie{g}} + Ad_{g}$ is singular precisely when $Ad_g$ is rotation by 180 degrees, which happens if and only if $g^2= - \id$, so the only possible singular value is $ (-\id)^{n+1}$. An example of a singular point in the preimage of $(-\id)^{n+1}$ is $(J,...J)$ for any $J \in G$ satisfying $J^2 = -\id$. 
\end{proof}

We will denote:
\begin{equation*}
X_n^s \co= X_{n}((-\id)^{n+1})
\end{equation*}
where s stands for singular.  Because
the set of regular values $G \setminus (-\id)^{n+1}$ is connected, it follows by the product neighbourhood theorem (see Milnor \cite{mi}) that
all the remaining fibres of $\Box_n$ are diffeomorphic. Indeed, $\Box_n$ must restrict to a trivial fibre bundle over the contractible set $G \setminus (-\id)^{n+1}$. We set

\begin{equation}\label{xnr}
X_n^r \co= X_{n}((-\id)^n)
\end{equation}
where $r$ stands for regular.

\begin{lem}
For $\mathcal{C} \subset G=SU(2)$ a conjugacy class other than $\{ \pm \id\}$, we have an isomorphism of $G$-spaces:
\begin{equation}\label{proddduct}
X_n(\mathcal{C}) \cong G/T \times X_n^r
\end{equation}
where the action on the right is the diagonal action. In particular $H_G( X_n(\mathcal{C})) \cong H_T(X_n^r)$. 
\end{lem}

\begin{proof}
The map $\Box_n : X_n(\mathcal{C}) \rightarrow \mathcal{C}$ is $G$-equivariant fibre bundle and the conjugacy class $\mathcal{C}$ is isomorphic as a $G$-space to $G/T$. It follows for general reasons that if $c \in \mathcal{C}$ is fixed by $T$, and $F_c$ is the fibre over $c$, then we have an isomorphism of $G$-spaces:
$$ X_n(\mathcal{C}) \cong  G \times_T F_c$$
where $G \times_T F_c$ is the $G$-space induced from the $T$-space $F_c$. Similarly $G/T \times X_n^r$ is induced from the $T$-space $X_n^r$. So it suffices to show that $X_n^r \cong F_c$ as $T$-spaces. 

Restricting $\Box_n: G^{n+1} \rightarrow G$ to the preimage of $ T \setminus (-\id)^{n+1}$ defines a deformation from $X_n^r$ to $F_c$ in the sense of Palais-Stewart \cite{ps}. By the rigidity of compact group actions on compact manifolds, that $X_n^r$ and $F_c$ are isomorphic $T$-spaces. 
\end{proof}

It may be helpful to consider some examples for small $n$.  The simplest examples are 

\begin{equation}\label{easy1}
X_0^r \cong S^0
\end{equation}
\begin{equation}\label{easy2}
X_0^s \cong S^2
\end{equation}
where $G$ acts trivially on $S^0$ and by the usual action on $S^2 \cong \C P^1$. 
The next example requires a little work. Identify $S^2$ with the adjoint orbit of $$  \left[ \begin{array}{cc}
i & 0  \\
0 & -i  \end{array} \right] \in \lie{su}(2)$$ and identify $S^1 \cong \R / 2 \pi \Z$. Consider the map $f: S^2 \times S^1 \rightarrow SU(2)$ defined by $ f( X, t) = exp(tX)$.  This map is $G$-equivariant, where $G$ acts trivially on the $S^1$ factor and by the adjoint action on the $S^2$ factor.

\begin{prop}\label{hardyharhar}
Define $ F \co S^2 \times S^1 \rightarrow SU(2)^2$ by $F(X,t) = (f(X,t),f(X,-t +\pi/2))$. Then $F$ restricts to a $SU(2)$-equivariant diffeomorphism
\begin{equation*}
X_1^r \cong S^2 \times S^1
\end{equation*}
where we view $X_1^r$ as a subset of $SU(2)^2$.
\end{prop}

\begin{proof}
Equivariance is immediate. We have $exp(tX)^2 exp((-t+\pi/2)X)^2 = exp( \pi X) = -\id$, so $f$ maps into $X_1^r$. The tangent map $df_{(X,t)}$ is injective for $t \notin \pi \Z$, so it follows that $F$ is an immersion. Injectivity and surjectivity are straightforward to verify.
\end{proof}

These examples, though helpful for illustrating the general formulas of Theorem \ref{onepoly}, are perhaps deceptively simple. We show in Proposition \ref{lettieri} that all remaining examples have torsion in integral homology.

\subsection{$X_n(\epsilon) \rightarrow SU(2)^n$ is a cohomological Galois cover}\label{almostfinshed}

For $g\in G = SU(2)$, define $\sqrt{g} \co= \{ h \in G | h^2 =g \}$.

\begin{lem}\label{S2S0}
If $g \in G$, then up to diffeomorphism,

\begin{equation*}
\sqrt{g}\cong \begin{cases}
S^2 &\text{ $g = -\id$}\\
S^0 &\text{ otherwise }\\
\end{cases}.
\end{equation*}
\end{lem}

\begin{proof}
Straightforward.
\end{proof}

Let $Z(G) = \{ \pm \id\}$ denote the centre of $G$.

\begin{lem}\label{couscous}
For $g \in G$, $Z(G) \cong \Z_2$ acts on $\sqrt{g}$ by right multiplication and $$H(\sqrt{g} / Z(G); \Q) \cong H(pt;\Q).$$
\end{lem}

\begin{proof}
Certainly the action is well defined because $h^2 = (-h)^2$, for $h \in G$.  From Lemma \ref{S2S0}, it follows that,

\begin{equation*}
\sqrt{g}/Z(G)=\begin{cases}
\R P^2 &\text{ $g = -\id$}\\
pt &\text{ otherwise }\\
\end{cases}
\end{equation*}
both of which have the cohomology of a point rationally.
\end{proof}

\begin{prop}\label{somemanyprops}
For $\epsilon = \pm \id$, let $\rho : X_n(\epsilon) \rightarrow G^n$ denote projection onto the last $n$ factors, $$\rho( (g_0,...,g_n)) = (g_1,...,g_n).$$ Let $Z(G) \cong \Z_2$ act on $X_n(\epsilon)$ by multiplying the $0$th factor. Then $( \rho: X_n(\epsilon) \rightarrow G^n, \Z_2)$ forms a strong cohomological Galois cover (Definition \ref{princ})  for rational cohomology. In particular the invariant subring $H(X_n(\epsilon))_+ = H(X_n(\epsilon))^{\Z_2} \cong H(G^n)$.
\end{prop}

\begin{proof}
The map $\rho$ is closed because $X_n(\epsilon)$ is compact. The $\Z_2$ action clearly leaves $\rho$ invariant. It only remains to study the fibres of $\rho$.  

Recall that $\Box_n(g_0,...,g_n) = \prod g_i^2$. For $ (g_1,...,g_n) \in G^{n}$,

\begin{align*}
\rho^{-1}((g_1,...,g_n)) = \{ (g_0,...,g_n) \in G^{n+1} | \Box_n(g_0,...,g_n) = \epsilon \} \cong \\
 \{ g_0 \in G | (g_0)^2 = \epsilon \Box_{n-1}(g_1,...,g_n)^{-1} \} = \sqrt{ \epsilon \Box_{n-1}(g_1,...,g_n)^{-1}}
\end{align*}
so by Lemma \ref{couscous}, $H(\rho^{-1}((g_1,...,g_n)) / \Z_2) \cong H(pt)$.
\end{proof}

\begin{rmk}\label{MIA}
In \cite{b2} chapter 8, a deeper explanation for Proposition \ref{somemanyprops} is provided. It is shown that in the case of a general compact, simply connected group $K$, the projection map is a ``cohomological covering map" with deck transformation group $Tor_1^{\Z}(\Z_2, Z(K))$ which fails to be Galois.
\end{rmk}

\subsection{Betti numbers of $X_{n}^r$ and $X_n^s$}\label{bettys}

In this section we prove

\begin{thm}\label{onepoly}
For rational cohomology, the Poincar\'e polynomials for the representation varieties $X_n^r$ and $X_n^s$ are
\begin{equation}\label{numbe1}
P_t(X_n^r) = P_t( X_n((-\id)^{n}) = (1+t^3)^n + (t+t^2)^n
\end{equation}
and
\begin{equation}\label{numbe2}
P_t(X_n^s) = P_t( X_n((-\id)^{n+1})= (1+t^3)^n + t^2(t +t^2)^n.
\end{equation}
\end{thm}

The cup product structure will be determined in \S \ref{local}.

\begin{rmk}
Though we work throughout with rational coefficients, Theorem \ref{onepoly} remains true in odd characteristic. 
\end{rmk}

\begin{rmk}
By Proposition \ref{somemanyprops}, $H^*(X_n^r)_+ \cong H^*(X_n^s)_+ \cong H^*(G^n)$ and $P_t(G^n) \cong (1+t^3)^n$.  It follows that the first and second terms in (\ref{numbe1}) and (\ref{numbe2}) are the Poincar\'e polynomial for the +1 and -1 eigenspaces of the $\Z_2$ action, respectively.
\end{rmk}

\begin{cor}\label{torefer}
For $\epsilon \in \{\pm \id\}$ the $T$-space $X_{n}(\epsilon)$ is equivariantly formal over rational coefficients.
\end{cor}
\begin{proof}
Using Proposition \ref{forcrit}, we need only show that $$ \dim H( X_n(\epsilon)) = \dim H(X_n(\epsilon)^T).$$ From Lemma \ref{twotoris} and Theorem \ref{onepoly} we find that both sides equal $2^{n+1}$.
\end{proof}

We used above the following basic result of about torus actions due to Borel:

\begin{prop}\label{forcrit}( see \cite{bor} IV 5.5)
Let $T$ be a compact torus and let $X$ be a compact Hausdorff $T$-space. Then $$\dim H^*(X) \geq \dim H^*(X^T)$$ with equality if and only if $X$ is $T$-equivariantly formal.
\end{prop}

The proof of Theorem \ref{onepoly} is by induction, using the long exact sequences (\ref{les2}) and (\ref{hohum}).

The trace map $Tr: G \rightarrow \R$,
$$Tr( \left[ \begin{array}{cc}
a & b  \\
-\bar{b} & \bar{a}  \end{array} \right] ) = a + \bar{a}$$ describes the height function for an imbedding of $G \cong S^3$ in $\R^4$. The image of $Tr$ is $[-2,2]$ and its only singular points are``poles" $\{ \pm \id\}$ where it achieves
its extreme values $\{ 2, -2\}$. It follows from Proposition \ref{singu} that the function
$f\co= (-1)^n Tr \circ \Box_n \co G^{n+1} \rightarrow \R$ has singular loci
$X_n^r$ and $X_n^s$ where $f$ achieves its maximum and minimum values respectively. 

The real algebraic map $f$ is by definition a \textbf{rug function} for its minimizing variety $X_n^s \subset G^n$.
By a result of Durfee \cite{du}, the inclusion $$X_n^s
\hookrightarrow f^{-1}([-2,2))= G^{n+1} - X_n^r$$ is a homotopy equivalence. We
deduce using Poincar\'e duality that:

\begin{equation}\label{step1}
H^k( G^{n+1}, X_n^r) \cong H_{3n+3-k}(X_n^s)
\end{equation}
We will use (\ref{step1}) in combination with the long exact sequence of the pair $i\co X_n^r \hookrightarrow G^{n+1} $:

\begin{equation}\label{les2}
...\rightarrow  H^*(G^{n+1}, X_n^r) \rightarrow H^*(G^{n+1}) \rightarrow^{i^*}
H^*(X_n^r) \rightarrow ...
\end{equation}
to relate $H(X_n^s)$ and $H(X_n^r)$. We now describe the image of $i^*$.

Let $\rho'\co G^{n+1} \rightarrow G^n$ denote projection onto all but the zeroth factor and let $\rho\co X_{n}^r \rightarrow G^n$ denote
the restriction of $\rho'$ to $X_{n}^r$.  We have a
commutative diagram:
\begin{equation}\begin{CD}\label{comm}
\xymatrix{X_{n}^r \ar[d]^{\rho} \ar[r]^i& G^{n+1} \ar[d]^{\rho'}\\
G^n  \ar@{=}[r] &   G^n}
\end{CD}\end{equation}
where $\rho: X_n^r \rightarrow G^n$ is a cohomological Galois cover by Proposition \ref{somemanyprops}.

\begin{lem}\label{tauinv}
The image of the map   $$i^*\co H(G^{n+1}) \rightarrow H(X_n^r)$$ appearing in (\ref{les2}) is the $\Z_2$-invariant subring $H(X_n^r)^{\Z_2} \cong H(G^n)$.
\end{lem}
\begin{proof}
The $\Z_2$-action on $X_n^r$ extends in an obvious way to $G^{n+1}$ where it is isotopically trivial.
Consequently, the image of $i^*$ lies in $H(X_n^r)^{\Z_2}$.  On the other hand, because $\rho' \co G^{n+1} \rightarrow G^n$ induces an injection in cohomology, $\im (i^*) \supseteq \im (\rho' \circ i)^* = im(\rho^*) = H(X_n^r)^{\Z_2}$, where the last equality follows from Theorem \ref{cov}, completing the proof.
\end{proof}

We now describe the other long exact sequence we will need for the induction argument of Theorem \ref{onepoly} .
For $n \geq 1$, consider the real algebraic map,
\begin{equation*}
\phi\co X_{n}^r \rightarrow [-2,2]
\end{equation*}
defined by $\phi(g_0,...,g_n) = Tr( g_0)$.

\begin{lem}\label{iminyork}
$\phi$ has two singular values, \{2,-2\}, over which $\phi$ has fibres isomorphic to $X_{n-1}^s$.
\end{lem}
\begin{proof}
The map $\phi = Tr \circ \pi_0$ where $\pi_0 \co X_n^r \rightarrow G$ is projection onto the zeroth factor. Because the singular values of $Tr : G \rightarrow \R$ are $\{\pm 2\}$, and $Tr^{-1}(\pm 2) = \pm \id$, it will be sufficient to show that  $sing(\pi_0)$ is a subset of $\{\pm \id\}$, where we denote by $sing(\pi_0)$ the singular value set of $\pi_0$.

$X_n^r$ fits into the pullback diagram:

\begin{equation*}\begin{CD}
\xymatrix{  X_n^r \ar[r]^{\rho} \ar[d]^{\pi_0} &  G^n \ar[d]^{\Box_{n-1}} \\
            G  \ar[r]^f &   G }
\end{CD}\end{equation*}
where $f \co G \rightarrow G$ sends $f(g) = (-1)^n g^{-2}$.  By the constant rank theorem, regular values pullback to regular values, so
\begin{equation*}
sing(\pi_0) \subseteq f^{-1}(sing(\Box_{n-1})) = f^{-1}( (-\id)^n ) = \pm \id
\end{equation*}
where we input Proposition \ref{singu} to get the middle inequality. The fibres of $\phi$ over $\pm \id$ are identified via the pull back diagram with $\Box_{n-1}^{-1} ( (-1)^n \id) = X_{n-1}^s$, completing the proof.
\end{proof}
We will consider the long exact sequence in cohomology for the pair $(X_n^r, \phi^{-1}(2))$.

\begin{equation*}
...\rightarrow H^*(X_n^r, \phi^{-1}(2)) \rightarrow H^*(X_n^r) \rightarrow H^*(\phi^{-1}(2)) \rightarrow ... .
\end{equation*}
Certainly $H^*(\phi^{-1}(2)) \cong H^*(X_{n-1}^s)$.  Furthermore, $\phi$ forms a rug function for $\phi^{-1}(-2)$ in the sense of \cite{du}, so the inclusion $\phi^{-1}(-2) \hookrightarrow \phi^{-1}([-2,2)) = X_n^r - \phi^{-1}(2)$ is a homotopy equivalence. By Poincar\'e duality,

\begin{equation}\label{threerrrs}
H^d(X_n^r,\phi^{-1}(2)) \cong H_{3n-d}(X_{n-1}^s)
\end{equation}
so we get an exact sequence for every degree $d$:

\begin{equation}\label{hohum}
H_{3n-d}(X_{n-1}^s) \rightarrow H^{d}(X_n^r) \rightarrow H^d(X_{n-1}^s).
\end{equation}

The following Lemma will form part of the induction step in the proof of Theorem \ref{onepoly}:

\begin{lem}\label{ineqqq}
Suppose that $\dim H(X_{n-1}^s) = 2^{n}$. Then the exact sequences (\ref{hohum}) extend to short exact sequences

\begin{equation*}
0 \rightarrow H_{3n-d}(X_{n-1}^s) \rightarrow H^{d}(X_n^r) \rightarrow H^d(X_{n-1}^s) \rightarrow 0
\end{equation*}
for all d.
\end{lem}

\begin{proof}
If $\dim H(X_{n-1}^s) = 2^n$, then the exact sequences (\ref{hohum}) imply that $\dim H(X_{n}^r ) \leq 2^{n+1}$ with equality if and only if they extend to short exact sequences.

On the other hand, by Proposition \ref{forcrit} we have a lower bound $ \dim H(X_{n}^r) \geq \dim H((X_{n}^r)^T)$, where $(X_{n}^r)^T$ is the $T$-fixed point set under the conjugation action of $T$ on $X_{n}^r$. We showed in Lemma \ref{twotoris} that $(X_{n}^r)^T$ is isomorphic to two disjoint copies of $T^n$.  In particular $\dim H((X_{n}^r)^T) = 2^{n+1}$ completing the proof.
\end{proof}

Strictly speaking, Proposition \ref{forcrit} is true for Cech cohomology with rational coefficients which coincides with singular cohomology in our case.

We now have all the preliminary results necessary to compute the Betti numbers. Recall that for a topological space $X$, the (ordinary, rational) \emph{Poincar\'e polynomial} $P_t(X) = \sum_{i=0}^{\infty} b_it^i$, where $b_i = \dim H^i(X)$ is the $i$th Betti number of $X$.

\begin{proof}[Proof of Theorem \ref{onepoly}]
The proof uses induction on $n$.  First of all, $X_0^r \cong S^0$ and $X_0^s \cong S^2$ so the theorem holds in this case.

Assume now that the theorem holds when $n=k-1$ where $k\geq 1$, so in particular $P_t(X_{k-1}^s) \cong (1+t^3)^{k-1} + t^2(t+t^2)^{k-1}$ and $\dim H(X_{k-1}^s) = P_1(X_{k-1}^s) = 2^n$. By Lemma \ref{ineqqq} we get

\begin{equation*}
P_t(X_{k}^r) = P_t(X_{k-1}^s) + t^{3k}P_{t^{-1}}(X_{k-1}^s) = (1+t^3)^{k} + (t+t^2)^{k}
\end{equation*}
as desired. It remains to determine $P_t(X_k^s)$.  We deduce from the long exact sequence (\ref{les2}) and Lemma~\ref{tauinv} that:

\begin{equation*}
P_t(G^{k+1}, X_k^r)= P_t(G^{k+1}) + tP_t(X_k^r) -(1+t)P_t(G^k) = t^3(1+t^3)^k + t(t+t^2)^k
\end{equation*}
which by (\ref{step1}) is equivalent to
\begin{equation*}
P_t(X_k^s) = t^{k+3}P_{t^{-1}}(G^{k+1}, X_k^r) = (1+t^3)^k + t^2(t +t^2)^k
\end{equation*}
completing the induction.
\end{proof}



\subsection{Characteristic 2 coefficients}\label{rockthecasbah}
We saw at the end of \S \ref{Regular and singular}, that for $X_0^r$, $X_0^s$ and $X_1^r$ the cohomology is torsion free so in these cases the Poincar\'e polynomial in Theorem \ref{onepoly} remains valid for characteristic 2 coefficients. The next proposition shows that it does not for the remaining cases.

\begin{prop}\label{lettieri}
The homology groups $H_2(X_n^r, \Z)$ for $n \geq 2$ and $H_2(X_n^s, \Z)$ for $n \geq 1$ contain $2$-torsion.
\end{prop}

\begin{proof}
Let $\rho' \co G^{n+1} \rightarrow G^n$ denote the projection map of Proposition \ref{somemanyprops}, and the degree 2 map $\rho\co X_n^r \rightarrow G^n$ its restriction under the inclusion $i \co X_n^r \hookrightarrow G^{n+1}$. This induces a map in cohomology:

\begin{equation*}\begin{CD}
H^{3n}(G^{n+1},\Z) @> i^* >> H^{3n}(X_n^r, \Z) \cong \Z \\
 @AA \rho'^*A             @AA \rho^* = 2 A \\
H^{3n}(G^n, \Z)   @>>>   H^{3n}(G^n, \Z) \cong \Z
\end{CD}\end{equation*}
Thus those generators in the image of $\rho'^*$ map to $2 H^{3n}(X_n^r, \Z)$ under $i^*$. Of course we can replace $\rho$ in this argument with projection onto any other set of factors, so for $n \geq 1$ we can show that $\im(i^*) \subset 2 H^{3n}(X_n^r, \Z)$.  Thus the LES of the pair $ i \co X_n^r \hookrightarrow G^{n+1}$ gives rise to two torsion in $ H^{3n+1}(G^{n+1}, X_n^r ; \Z)$, which by (\ref{step1}) is isomorphic to $H_{2}(X_n^s ; \Z)$.  This completes the proof for $X_n^s$ and $n \geq 1$.

For the remaining case, consider the LES in homology for the pair $(X_{n+1}^r, X_n^s)$. In particular, we have an exact sequence:

\begin{equation}\label{beckod}
H_1(X_{n+1}^r, X_n^s ; \Z) \rightarrow H_2(X_n^s; \Z)  \stackrel{\zeta}{\rightarrow} H_2(X_{n+1}^r; \Z)
\end{equation}
and by (\ref{threerrrs}), $H_1(X_{n+1}^r, X_n^s ; \Z) \cong H^{3n+2}(X_n^s)$. The smooth locus of $X_n^s$ has dimension $3n$ while the singular locus is isomorphic to $(S^2)^{n+1}$ and has dimension $2(n+1)$.  Consequently, for $n \geq 1$, $H^{3n+2}(X_n^s)=0$ and so the map $\zeta$ of (\ref{beckod}) is an inclusion, which gives rise to $2$-torsion in $H_2(X_{n+1}^r; \Z)$ for $n+1 \geq 2$.
\end{proof}

\begin{rmk}
One of the most common methods of proving equivariant formality for a torus action on a manifold, is to produce a Morse-Bott function whose critical points coincide with the fixed points. Proposition \ref{lettieri} precludes the existence of such a Morse function for $X_n^r$ when $n > 1$ because the fixed point set $X_n^{rT}$ is torsion free and so such a function would contradict the Morse inequalities in characteristic $2$.
\end{rmk}

\subsection{Proof of Theorem \ref{ale}}\label{morgau}

Recall notation from \S \ref{resout} . Let $\mathcal{A}$ be the space of connections on a trivial principal $G$-bundle over the nonorientable surface $\Sigma_n$, with gauge group $\gau$.

\begin{proof}[Proof of Theorem \ref{ale}]
The projection map $X_n(+\id) \rightarrow G^n$ is isomorphic to a map $Hom(\pi_1(\Sigma_n),G) \rightarrow Hom(\pi_1(\Sigma^o_n),G)$ induced by the inclusion of surfaces $\Sigma^o_n \hookrightarrow \Sigma_n$ where $\Sigma_n$ be a connected sum of $n+1$ copies of $\R P^2$ and $\Sigma^o_n = \Sigma_n - M$ is a surface with boundary obtained by removing a Mobius band which is the tubular neighbourhood of a loop representing the $0$th generator of the fundamental group of $\Sigma_n$. 

The inclusion $j \co \Sigma_n^o \hookrightarrow \Sigma_n$ induces an isomorphism in rational cohomology $H(\Sigma_n^o;\Q) \cong H(\Sigma_n;\Q)$ so by the methods of Atiyah-Bott (\cite{ab2} sect.2), $j$ induces an isomorphism $H(B\mathcal{G}^o) \cong H(B\mathcal{G})$. 

Finally, because $\Sigma_n^o$ deformation retracts onto a wedge of circles, restriction induces an isomorphism $H(B\mathcal{G}^o)\cong H_{\mathcal{G}^o}(\mathcal{A}_{\Sigma^o}) \cong H_{\mathcal{G}^o}(\mathcal{A}_{\Sigma^o}^{flat}) \cong H_G(G^n)$. To summarize we get a commutative diagram

\begin{equation*}\begin{CD}
 H(B\mathcal{G}) @> \kappa >> H_{\mathcal{G}}(\mathcal{A}^{flat}_{\Sigma}) \cong H_G( X_n(+\id))\\
   @AA\cong A                                      @AA\rho^*A \\
H(B\mathcal{G}^o) @>\cong >> H_{\mathcal{G}^o}(\mathcal{A}^{flat}_{\Sigma^o}) \cong H_G(G^n).
\end{CD}\end{equation*}
Thus $\kappa$ is injective if and only if $\rho^*$ is. This was shown for $SU(2)$ in Proposition \ref{somemanyprops}, and for general compact simply connected G in Remark \ref{MIA}.

\end{proof}


\subsection{Relationship between $X_n^r$ and $X_n^s$}\label{relationships1}

Recall (\ref{easy2}), that $X_0^s \cong S^2$ and that $\Z_2$ acts via the antipodal map on $X_0^s$. Let $\Z_2$ act diagonally on $X_n^r \times X_0^s$ sending $(g_0,...,g_n) \times (h) \rightarrow (-g_0,..,g_n) \times (-h)$, and denote the orbit space by $X_n^r \times_{\Z_2} X_0^s$.
In this section we prove 

\begin{prop}\label{mikei}
There is an isomorphism
\begin{equation}\label{mikej}
H_T(X_n^r \times_{\Z_2} X_0^s) \cong H_T(X_n^s).
\end{equation}
Furthermore,  using the diffeomorphism between fixed point set described in (\ref{montana}) we obtain a commutative diagram

\begin{equation}\begin{CD}\label{mikek}
H_T(X_n^r \times_{\Z_2} X_0^s) @>>> H_T((X_n^r \times_{\Z_2} X_0^s)^T)\\
 @VV^{\cong}V                             @VV^{\cong}V\\
H_T(X_n^s)             @>>>   H_T((X_n^s)^T)
\end{CD}\end{equation}
where the horizontal arrows are the fixed point localization maps induced by inclusion. A similar diagram is valid for ordinary cohomology. 
\end{prop}
The isomorphism (\ref{mikej}) is not induced by a direct map of spaces. Rather, we introduce a third space $Z$ and an equivariant correspondence diagram

\begin{equation}\begin{CD}
\xymatrix{ & Z  \ar[dr] \ar[dl]^{\phi} &  \\
X_n^s&   & X_n^r \times_{\Z_2} X_0^s }
\end{CD}\end{equation}
which induces (equivariant and ordinary) cohomology isomorphisms in both directions and whose restrictions to fixed points are diffeomorphisms.

Define the subvariety $ Z \subset X_n^r \times_{\Z_2} X_0^s$ by
$$ Z = \{ (\pm g_0,...,g_n) \times (\pm h) | g_0h = hg_0\}.$$ 
We obtain a surjective map
$$ \phi : Z \rightarrow X_n^s$$
defined by $ \phi ( \pm g_0,..., g_n) \times ( \pm h)) = (h g_0,..., g_n)$. One may readily verify that:

\begin{equation}
\phi^{-1}(g_0,...,g_n) \cong  \begin{cases} S^2/\Z_2 \cong \R P^2 \text{       ,  if  $g_n = \pm \id$}\\
					S^0 /\Z_2 \cong pt \text{      ,   otherwise}
					\end{cases}
\end{equation}
So by the same reasoning as in Proposition \ref{somemanyprops}, we get
\begin{equation}\label{blur10}
 H_T(X_n^s) \cong H_T(Z).
\end{equation}
which implies similar isomorphisms in ordinary cohomology.
Now consider the inclusion $ Z \hookrightarrow X_n^r \times_{\Z_2} X_0^s$.  The fixed point set 

\begin{equation}\label{blur11}
(X_n \times_{\Z_2} X_0^s)^T = \{ (\pm t_0, ..., t_n)\times (\pm h)| t_i, h \in T\} \cong (X_n^r)^T
\end{equation}
is contained in $Z$. In particular, the localization map $H_T( X_n \times_{\Z_2} X_0^s) \rightarrow H_T( (X_n \times_{\Z_2} X_0^s)^T)$ factors through $H_T(Z)$ and is injective by equivariant formality. Comparing Betti numbers, we get 

\begin{equation}\label{blur12}
  H_T(X_n^r \times  X_0^s)^{\Z_2} \cong H_T(X_n^r \times_{\Z_2} X_0^s) \cong H_T( Z)
\end{equation}
and a similar isomorphism for ordinary cohomology.
\begin{proof}[Proof of Proposition \ref{mikei}]

Equation (\ref{mikej}) follows from (\ref{blur10}) and (\ref{blur12}). Commutativity of (\ref{mikek}) follows from (\ref{blur11}). 
\end{proof}

To better exploit this result, we will make use of a lemma.

\begin{lem}\label{urbano}
Let $X$, $Y$ be equivariantly formal $T$-spaces. Then $X \times Y$ equipped with the diagonal action is also equivariantly formal and $$ H_T(X \times Y) \cong H_T(X) \otimes_{H(BT)} H_T(Y) $$
as modules over $H(BT) = H_T(pt)$.
\end{lem} 

\begin{proof}
Because $H_T(X)$ and $H_T(Y)$ are free modules over $H(BT)$, the result follows by an Eilenberg-Moore spectral sequence argument.
\end{proof}
Combining Proposition \ref{mikei} with Lemma \ref{urbano} we obtain

\begin{equation}\label{urbantwo}
 H_T(X_n^s) \cong (H_T(X_n^r)_+ \otimes_{H(BT)} H_T(X_0^s)_+) \oplus (H_T(X_n^r)_- \otimes_{H(BT)} H_T(X_0^s)_-).
\end{equation}

\subsection{Relationship between $X_n^r$ and $X_1^r$}\label{relationships2}

In this section, we state a factorization theorem for the (ordinary and equivariant) cohomology of $X_n^r$ and $X_n^s$. The proof is postponed until \S \ref{chapter6}.

\begin{lem}\label{wat}
For $n\geq 1$ there is an isomorphism 
\begin{equation}
H_T( X_{n}^r) \cong H_T( X_{n-1}^r \times_{\Z_2} X_1^r)
\end{equation}
which is natural with respect to localization. A similar isomorphism holds for ordinary cohomology.
\end{lem}
Lemma \ref{wat} is equivalent to (\ref{reduc}) from the introduction. By iterating Lemma \ref{wat} and using Proposition \ref{mikei}, we may reduce the general case to that of $n=1$. 

\begin{thm}\label{toref}
There are isomorphisms
$$ H_T(X_n^r) \cong H_T(X_1^r \times_{\Z_2} ... \times_{\Z_2}X_1^r)$$
and
$$ H_T(X_n^s) \cong H_T( X_1^r \times_{\Z_2} ... \times_{\Z_2} X_1^r \times_{\Z_2} X_0^s)$$
which are natural with respect to fixed point localization. Similar isomorphisms hold for ordinary cohomology.
\end{thm}

To see why we might expect Lemma \ref{wat} to be true, consider the following informal argument. 
Let  $$\pi_{n}: X_n^r \rightarrow SU(2)$$ be the projection map sending $(g_0,...,g_n)$ to $g_n$. As explained in Lemma \ref{iminyork}
, $\pi_n$ has fibres over $\pm \id$ isomorphic to $X_{n-1}^s$, and restricts to a trivial $X_{n-1}^r$ fibration over $ SU(2) \setminus \{ \pm \id\}$. 

Consider now the map $$ \kappa: X_{n-1}^r \times_{\Z_2} X_1^r \rightarrow SU(2)$$ obtained by projection $\kappa((\pm g_0,..., g_n)\times (\pm h_0, h_1)) = h_1$. The projection map factors through $ X_1^r/\Z_2$ which has fibres $X_{n-1}^r$ so by Proposition \ref{hardyharhar}, we see that $\kappa$ has fibres $ X_{n-1}^r \times_{\Z_2} S^2$ over $\pm \id$, and restricts to a trivial $ X_{n-1}^r$ fibration over $SU(2) \setminus \{\pm \id\}$. 

In view of Proposition \ref{mikei}, one might try to prove Lemma \ref{wat} by replacing the singular fibres of $\kappa$ with the singular fibres of $\pi_n$. Unfortunately, we have been unsuccessful in turning this informal argument into a formal proof. Instead, in the next section, we will compute the image of the localization map for $X_{n+1}^r$ using induction on $n$ and then compare with the image of the localization map for $X_{n}^r \times_{\Z_2} X_1^r$ under the identification $$ (X_{n+1}^r)^T \cong (X_{n}^r)^T \times_{\Z_2} (X_1^r)^T$$ described in (\ref{montana}) .

\section{ The localization map for $T$ action}\label{local}\label{chapter6}

In this Chapter we compute the image of the localization map in equivariant cohomology,

\begin{equation*}
i^*\co H_{T}(X) \rightarrow H_{T}(X^{T})
\end{equation*}
where $X= X_n^r$ or $X_n^s$, and $i\co X^{T} \hookrightarrow X$ is the inclusion of the fixed point set. As in Chapter \ref{chapter5}, we set $G = SU(2)$, and $T\subset G$ the maximal torus of diagonal matrices with Lie algebras $\lie{g}$ and $\lie{t}$ respectively. We will work with rational coefficients, though all the results can be extended to odd characteristic with more work.

The action of $\Z_2$ on $X$ defined in \S \ref{almostfinshed} commutes with the $T$ action, and so $i^*$ decomposes into a direct sum $i_+^* \oplus i_-^*$ :

\begin{equation*}
i^*_+ \co H_T(X)_+ \rightarrow H_T(X^T)_+ \cong H(T^n)\otimes H(BT)
\end{equation*}
and
\begin{equation*}
i^*_- \co H_T(X)_- \rightarrow H_T(X^T)_- \cong H(T^n)\otimes H(BT)
\end{equation*}
by (\ref{nearlydoneyeah2}), where we use subscripts $\pm $ to denote the invariant and skewinvariant weight spaces.

Our strategy will be to consider the invariant and skewinvariant parts of $i^*$ separately, where $i^*_+$ can be understood directly and $i^*_-$ requires an induction argument. We will find that not only does $i^*$ respect the $\Z_2\oplus \Z$ grading on $H_T(X^T)$, but also the $\Z_2 \oplus \Z^2$ grading, proving Corollary \ref{bigradient}.

\begin{rmk}
Because $Z(G)$ acts trivially, it may seem more natural to consider instead the effective conjugation action by $T/Z(G)$. However the only differences in cohomology will be 2-torsion and since we work with rational coefficients it won't make any difference.
\end{rmk}

\subsection{Invariant part}\label{invtpart}
Let $X$ denote one of representation varieties $X_n^r$ or $X_n^s$ defined in \S $\ref{Regular and singular}$. As was pointed out in (\ref{tauinv}), the projection map $\rho$ defines an isomorphism between $H(X)_+ = H(X)^{\Z_2}$ and $H(G^n)$.  Because $\rho$ is $T$-equivariant we deduce

\begin{equation*}
\rho^*\co H_{T}(G^n) \cong H_{T}(X)_+.
\end{equation*}
By (\ref{deckd}), the restriction of $\rho$ to $X^T$ determines a trivial double cover $\rho_T\co X^T \rightarrow T^n$. Thus

\begin{equation*}
\rho^*_T\co H_{T}(T^n) \cong H_{T}(X^T)_+.
\end{equation*}
This fits into the commutative diagram:

\begin{equation}\begin{CD}\label{werter}
\xymatrix{ H_{T}(X)_+  \ar[r]^{i^*} & H_{T}(X^T)_+ \\
       H_{T}(G^n) \ar[r]^{j^*}\ar[u]^{\rho^*}_{\cong} &  H_{T}(T^n)\ar[u]^{\rho^*_T}_{\cong} }
\end{CD}\end{equation}
where $j\co T^n \hookrightarrow G^n$ is the inclusion map and the vertical maps are isomorphisms. So we only need understand the image of the localization map $j^*$.

Consider first the case $n=1$. Recall that $H(BT) \cong \Q [c_1]$ is a polynomial algebra with a single generator in degree 2.

\begin{lem}\label{pattykingcole}
For $T = S^1$ acting on $SU(2)$ by conjugation, the image of the localization map $j^*: H_T(SU(2)) \rightarrow H_T(T) = H(T) \otimes H(BT)$ has image equal to $$\im(i^*) = \bigoplus_{p \leq q} H^p(T) \otimes H^{2q}(BT) = (H^0(T) \otimes H(BT)) \oplus (H^1(T) \otimes c_1  H(BT)).$$
\end{lem}

\begin{proof}
Because the conjugation action of $T$ on $SU(2)$ is equivariantly formal, the Betti numbers imply that the cokernel of $j^*$ is 1-dimensional of degree 1. Since $H_T^1(T) = H^1(T) \otimes H^0(BT)$ is 1 dimensional, the image of $j^*$ must be as stated.
\end{proof}

\begin{lem}
Let $T = S^1$ act diagonally on $SU(2)^n$ via the conjugation action.
The image of $j^*: H_T(SU(2)^n) \rightarrow H(T^n)\otimes H(BT)$ is equal to $ \bigoplus_{k\leq l} H^k(T^n) \otimes H^{2l}(BT)$.
\end{lem}

\begin{proof}
This follows from Lemma \ref{pattykingcole} by the equivariant Kunneth theorem, Lemma \ref{urbano}.
\end{proof}

\begin{prop}\label{razersnik3}
The image of $i^*_+ \cong j^*$ in $H_T(X^T)_+ \cong H(T^n)\otimes H(BT)$ is equal to $ \bigoplus_{k\leq l} H^k(T^n) \otimes H^{2l}(BT)$.
\end{prop}

\subsection{Skewinvariant part}\label{antiinvariantpart}

In this section we compute the image of $i^*_-$ for both the regular and singular representation varieties. We introduce subscripts $i_r$, $i_s$ to distinguish the cases. The following hold for values $n= 0,1,2,...$

\begin{prop}\label{razersnik}
The image of the map $$i^*_{r-} \co H_T(X_n^r)_- \rightarrow H_T(X^T)_- \cong H(T^n)\otimes H(BT)$$ is equal to $ \bigoplus_{k +l \geq n} H^k(T^n) \otimes H^{2l}(BT)$.
\end{prop}

\begin{prop}\label{razersnik2}
The image of the map $$i^*_{s-} \co H_T(X_n^s)_- \rightarrow H_T(X^T)_- \cong H(T^n)\otimes H(BT)$$ is equal to $ \bigoplus_{k +l \geq n+1} H^k(T^n) \otimes H^{2l}(BT)$.
\end{prop}

Our proof will use induction on $n$.  Let  $( \ref{razersnik}, n)$ and $(\ref{razersnik2},n)$ be the statements of Propositions \ref{razersnik} and \ref{razersnik2}, where we have made explicit the dependance on $n$.
\begin{lem}
For $n$ a nonnegative integer, (\ref{razersnik},n) implies (\ref{razersnik2},n).
\end{lem}

\begin{proof}
Let $j: (X_0^s)^T \rightarrow X_0^s$ be the fixed point inclusion, so that $$j_-^*: H_T(X_0^s) \rightarrow H_T(T^0) \cong H(BT).$$ 
Recall that $H(BT) = \Q [c_1]$ is a polynomial algebra in one degree 2 generator. Because $j_-^*$ is injective, it follows from the Betti numbers that $im( j_-^*) = c_1 H(BT) = H^{\geq 2}(BT)$.
By Proposition \ref{mikei} and (\ref{urbantwo}) we see that 

$$ im ( i_{s-}^*) = im( i_{r-}^*) \otimes  im(j^*_-) = c_1 im(i_{r-}^*)$$
completing the proof.
\end{proof}

Now observe that for each $j \in \{1,...,n\}$, we may define a $T$ invariant imbedding $$\tau_j: X_{n-1}^s \rightarrow X_n^r$$ defined by $\tau_j(g_0,...,g_{n-1}) = (g_0,... g_{j-1}, \id, g_j,..., g_{n-1})$. Each of the embeddings $\tau_j$ are equivariant for the $\Z_2$ action (Proposition \ref{somemanyprops}) so we obtain commutative diagrams:

\begin{equation}\begin{CD}
H_T(X_{n}^r)_- @>^{i^*_{r-}}>> H_T(X_n^{rT})_- @{=} H(T^n) \otimes H(BT)\\
@VV^{\tau_j^*}V			@VVV				@VV^{\sigma_j^*}V\\
H_T(X_{n-1}^s))_-  @>>^{i^*_{s-}}> H_T(X_{n-1}^{sT})_- @{=} H(T^{n-1}) \otimes H(BT)
\end{CD}\end{equation}
where we have identified the skewinvariant part of the cohomology of the fixed point set the cohomology of tori according to (\ref{montana}). Under this identification, the right most map is induced by the imbedding $\sigma_j: T^{n-1} \rightarrow T^n$, which takes $(t_1,...,t_{n-1})$ to $(t_1,...,t_{j-1},\id,t_j,...,t_{n-1})$.  

\begin{proof}

Thus the image of $i_{r-}^*$ satisfies: 
\begin{equation}\label{stuffstuff}
\im (i^*_{r-}) \subset \bigcap_{j=1}^n (\sigma_j^*)^{-1}( \im(i^* \circ \tau_j^*)) \subset   \bigcap_{j=1}^n (\sigma_j^{*})^{-1}( \im(i^*_{s-} ))
\end{equation}
and by induction hypothesis, we have $  \cap_{j=1}^n (\sigma_j^{*})^{-1}( \im(i^*_{s-} ))   \cong \oplus_{p+q \geq n} H^p(T^n) \oplus H^q(BT)$. Comparing Betti numbers, we deduce that the inclusions in  (\ref{stuffstuff}) must be equalities.
\end{proof}

\subsection{Proof of Lemma \ref{wat}}

We begin by rephrasing Lemma \ref{wat} in more explicit terms. Consider the correspondence diagram

\begin{equation}\begin{CD}
\xymatrix{ &(X_{n-1}^r\times_{\Z_2}X_1^r)^T \cong  X_n^{rT} \ar[dr]_{j_b} \ar[dl]^{j_a} \\
                   X_{n-1}^r \times_{\Z_2}X_1^r & &  X_n^r}
\end{CD}\end{equation}
where $j_a$ and $j_b$ are inclusions and the isomorphism follows from Lemma \ref{popeurban} and $(X_{n-1}\times_{\Z_2}X_1^r)^T \cong X_{n-1}^T\times_{\Z_2}X_1^{rT}$.  Our goal is to prove that the induced maps $j_a^*$ and $j_b^*$ in equivariant cohomology have the same image.

\begin{proof}

Let $i_k: X_{k}^{rT} \rightarrow X_{k}^r$ be inclusion. It follows from Lemma \ref{urbano} that   
\begin{equation}\begin{CD}
H_T(X_{n-1}^r \times_{\Z_2} X_1^r)_{\pm} @>^{\cong}>> H_T(X_{n-1}^r)_{\pm} \otimes_{H(BT)} H_T(X_1^r)_{\pm}\\
     @VV^{j_{a\pm}}V								@VV^{i_{(n-1)\pm} \otimes i_{1\pm}} V\\
 H_T(T^{n})   @>^{\cong}>>           H_T(T^{n-1}) \otimes_{H(BT)}H_T(T) 
\end{CD}\end{equation}

In particular, inputting the results of \S \ref{invtpart} and \S \ref{antiinvariantpart} 
\begin{align}
im(j_{a+}^*) = \oplus_{p\leq q, k\leq l} H^p(T^{n-1}) \otimes H^k(T) \otimes H^{q+l}(BT)\\
					= \oplus_{ p \leq q} H^p(T^n) \otimes H^q(BT)
\end{align}
while
\begin{align}
im(j_{a-}^*) = \oplus_{p+ q \geq n-1, k+ l\geq 1} H^p(T^{n-1}) \otimes H^k(T) \otimes H^{q+l}(BT)\\
					= \oplus_{ p +q \geq n} H^p(T^n) \otimes H^q(BT)
\end{align}
which completes the proof.
\end{proof}

\subsection{Bigrading and cup product}

The inclusion of the identity $\id \hookrightarrow T$ induces a natural map
$$ \psi: H_T(X) \rightarrow H(X)$$
from equivariant to ordinary cohomology. 

\begin{lem}\label{LHirsh}
Let $X$ denote the representation variety $X_n^r$ or $X_n^s$. The sequence:

\begin{equation*}\begin{CD}
0 \rightarrow c_1 H_T(X) \rightarrow H_T(X) @>\psi >> H(X) \rightarrow 0
\end{CD}\end{equation*}
is short exact.
\end{lem} 

\begin{proof}
Follows from Theorem \ref{formalityresult}. 
\end{proof}

Lemma \ref{LHirsh} allows us to describe the cup product structure of $H^*(X)$. The description is more transparent if we first introduce a second grading.

Recall that the equivariant cohomology ring of the fixed point set is endowed with $\Z^2$ grading via the Kunneth theorem: $$H_T^*(X^T) \cong H^*(X^T) \otimes H^*(BT)$$
with two variable equivariant Poincar\'e series $$ P_{x,y}^T(X^T) = P_x(X^T)P_y(BT)= 2(1+x)^n/(1-y^2) .$$
By Propositions \ref{razersnik3}, \ref{razersnik} and \ref{razersnik2} we see that the localization map respects this bigrading so,

\begin{lem}
For $X = X_n^r$ or $X_n^s$, the localization map  $i^*: H_T(X) \rightarrow H_T(X^T)$ is a morphism of $\Z_2 \oplus \Z^2$ graded rings.
\end{lem}
The kernel of $\psi$ in Lemma \ref{LHirsh} is also $\Z_2 \oplus \Z^2$-graded, so the ordinary cohomology, $H(X)$, also inherits this grading. The following should be compared with Theorem \ref{onepoly}.
\begin{prop}
As a $\Z^2$-graded ring, $H(X)$ has two variable Poincar\'e polynomial

\begin{equation}
P_{x,y}(X) = \begin{cases} 
		(1 + xy^2)^n + (x+y^2)^n & \text{ if $X = X_n^r$}\\
		(1 + xy^2)^n + y^2(x+y^2)^n & \text{ if $X = X_n^s$}\\
		\end{cases}
\end{equation}
where the even and odd components of the $\Z_2$-grading correspond to the first and second terms respectively.
\end{prop}
\begin{proof}

By Lemma \ref{LHirsh}, $H(X) \cong H_T(X) / c_1 H_T(X)$ as a bigraded ring. Using the injectivity of the localization map $i^*$, we write $$ H(X) \cong \im(i^*) / c_1 \im(i^*).$$
The result then follows by a simple calculation using the explicit description of $\im( i^*)$ from Propositions \ref{razersnik3}, \ref{razersnik} and \ref{razersnik2} .
\end{proof}

Recall from Proposition \ref{somemanyprops}, that the subring $H(X)_+$ is isomorphic to $H(G^n)$.

\begin{prop}
Let $X= X_n^r$ or $X_n^s$. Then in terms of the vector space decomposion into $\Z_2$-weight spaces $$\tilde{H}(X) = \tilde{H}(X)_+ \oplus H(X)_-$$
the cup product satisfies $$( a_1, b_1) \cup (a_2, b_2) = ( a_1 \cup a_2 + c(b_1,b_2), 0)$$
where the bilinear map $$c: H(X)_- \times H(X)_- \rightarrow H^{3n}(X)_+ \subset \tilde{H}(X)_+$$ is the Poincar\'e duality pairing when $X = X_n^r$ and is zero when $X = X_n^s$. 
\end{prop}

\begin{proof}
We address the case $X_n^r$, the other case being similar.  Notice that all elements of $H(X_n^r)_+$ have bidegree $(k, 2k)$ for some $k$, while elements in $H(X_n^r)_-$ have bidegree $( l, 2n -2l)$ for some $l$. 

Pairing an elements of degree (k, 2k) and (l, 2n-2l) ends up in (k+l, 2n -2l +2k) which must be zero unless $k=0$. Pairing elements of degree (k, 2n-2k) and (l, 2n-2l) ends up in (k+l, 4n -2(k+l)), which must be zero unless $4n - 2(k+l) = 2(k+l)$, and so must land in $(n, 2n)$ which is the top degree for the 3n dimensional manifold $X_n^r$.
\end{proof}

\section{The orbit space $X_n(\mathcal{\mathcal{C}})/SU(2)$}\label{chapter7}
 
In this section we compute the cohomology of the orbit space $ X_n(\mathcal{C})/SU(2)$. To accomplish this, we use a technique borrowed from Cappell-Lee-Miller \cite{clm}. We first describe the technique for a general equivariantly formal $SU(2)$-space $X$, and then apply it to our case $X_n(\mathcal{C})/SU(2)$. As usual we set $G= SU(2)$ and $T \subset SU(2)$ a maximal torus.

\subsection{General case}\label{gneyork}

Since $G = SU(2)$ has rank 1, every Lie subgroup of $G$ is either finite or contains a maximal torus of $G$. All maximal tori are conjugate in $G$ so we obtain:

\begin{lem}\label{bylemy}
Let $X$ be a $G=SU(2)$-space, let $G \cdot X^T$ be the union of orbits passing through the $T$-fixed point set $X^T$. Then the stabilizers of points in $G \cdot X^T$ have rank $1$ while the stabilizers of points in the complement $ X - G\cdot X^T$ are finite. 
\end{lem}

We are interested in the long exact sequence in equivariant cohomology for the pair $(X, G\cdot X^T)$. We call a $G$-space $X$ \textbf{nice} if is a finite CW complex.\footnote{More generally, we could work with compact Hausdorff spaces and Cech cohomology.}
 
\begin{prop}\label{yorkysit}
Let $X$ be a nice $G$-space. Then we have an isomorphism of long exact sequences:

\begin{equation*}\begin{CD}\label{toyourk}
 \xymatrix{ \ar[r]  & H_T(X,X^T)^W   \ar[r]  & H_T(X)^W   \ar[r]   & H_T(X^T)^W   \ar[r]  &   \\
   \ar[r]  & H_G(X,G \cdot X^T)    \ar[r] \ar[u]^{\cong} &      H_G(X)      \ar[r] \ar[u]^{\cong} &    H_G(G \cdot X^T)    \ar[r] \ar[u]^{\cong} &  }
\end{CD}\end{equation*}
\end{prop}

\begin{proof}
The morphism of long exact sequences (\ref{toyourk}) factors through the well known isomorphism

\begin{equation*}\begin{CD}\label{howyorkto}
\xymatrix{ \ar[r] & H_T(X,G \cdot X^T)^W  \ar[r] &   H_T(X)^W   \ar[r] &   H_T(G \cdot X^T)^W  \ar[r] &    \\
 \ar[r] &H_G(X,G \cdot X^T)   \ar[r]  \ar[u]^{\cong} &   H_G(X)     \ar[r]  \ar[u]^{\cong} &     H_G(G \cdot X^T) \ar[r]  \ar[u]^{\cong} &     }
\end{CD}\end{equation*}
so by the five lemma, we only need establish that the map $H_T(G \cdot X^T)^W  \rightarrow H_T(X^T)^W$ induced by inclusion of spaces is an isomorphism. This result was established in \S 4 of \cite{b1}.
\end{proof}

By Lemma \ref{bylemy}, $G$ acts with finite stabilizers on $X - G\cdot X^T$. Since we are working with rational coefficients we obtain (see \cite{bred}):

$$H_G( X, G \cdot X^T) \cong H(X/G, (G \cdot X^T)/G) = H(X/G, X^T/W).$$

Consequently, if we understand well the localization map $H_T(X) \rightarrow H_T(X^T)$, we can compute $H(X/G, X^T/W)$. For instance:

\begin{prop}\label{helpful}
Suppose a nice $G$-space $X$ is $T$-equivariantly formal for the restricted $T$ action. Then the Poincar\'e polynomial for the pair $(X/G, X^T/W)$ satisfies
\begin{equation}\label{poinpoly2}
P_t(X/G,X^T/W) = t[P_t^G(G \cdot X^T)-P_t^{G}(X)]
\end{equation}
and $H(X/G,X^T/W)$ has trivial cup product.
\end{prop}
\begin{proof}
By equivariant formality we have a short exact sequence $$0 \rightarrow H_T^*(X) \rightarrow H_T^*(X^T) \rightarrow H_T^{*+1}(X, X^T) \rightarrow 0$$ so by Proposition \ref{yorkysit} the sequence
$$0 \rightarrow H_G^*(X) \rightarrow H_G^*(G \cdot X^T) \rightarrow H_G^{*+1}(X, G \cdot X^T) \rightarrow 0$$
is also exact. Equation (\ref{poinpoly2}) follows.  The cup product on $H_G(X, G \cdot X^T) \cong H(X/G,X^T/W)$ is trivial because the boundary map in (\ref{howyorkto})
 surjects onto $H_G(X, G \cdot X^T)$.
\end{proof}

Next we would like to study the cohomology of the orbit space $X/G$ using the long exact sequence of the pair $(X/G, X^T/W)$. The following lemma is helpful.

\begin{lem}\label{prtolam}
Suppose that $X$ is a nice $G$-space . Then there is a morphism of long exact sequences:

\begin{equation*}\begin{CD}
\xymatrix{  \ar[r]&  H_T(X,X^T)^W  \ar[r] &  H_T(X)^W  \ar[r] & H_T(X^T)^W  \ar[r] & \\
\ar[r] & H(X/G, X^T/W)  \ar[r]  \ar[u]& H(X/G)\ar[r] \ar[u] &H(X^T)^W \ar[r] \ar[u] &\\}
\end{CD}\end{equation*}
where the map $H(X^T)^W \rightarrow H_T(X^T)^W = (H(X^T)\otimes H(BT))^W$, is induced by the inclusion $$H(X^T) = H(X^T) \otimes 1 \hookrightarrow H(X^T) \otimes H(BT).$$ 
\end{lem}

\begin{proof}
Straightforward.
\end{proof}

\subsection{The cohomology of $X_n(\mathcal{C})/G$}

In this section we apply the general strategy laid out in \S \ref{gneyork} to compute the Poincar\'e polynomial for $H(X/G)$ where $X = X_n(\mathcal{C})$ is an $G=SU(2)$-representation variety as defined in \S \ref{resout}. According to \S \ref{almostfinshed} these representation varieties must be isomorphic to one of $X_n^r$, $X_n^s$ or $X_n^r \times G/T$ as $G$-spaces, and we work with the latter notation.

\begin{lem}\label{dumbyorkys}
Let $X = X_n(\mathcal{C})$. Then the Poincar\'e polynomial for $H_G(G \cdot X^T)$ satifies:
\begin{equation}
P_t^G( G \cdot X^T) = \begin{cases}
\dfrac{(1+t)^n}{1-t^2}+\dfrac{(1-t)^n}{1+t^2} &\text{ if $X = X_n(\id)$}\\
\\
\dfrac{(1+t)^n}{1-t^2} &\text{ if $X = X_n(-\id)$}\\
\\
\dfrac{2(1+t)^n}{1-t^2} & \text{if $X \cong X_n^r \times G/T$}.
                     \end{cases}
\end{equation}
\end{lem}

\begin{proof}
By Proposition \ref{yorkysit} we know that $$H_G(G \cdot X^T) \cong H_T(X^T)^W \cong (H(X^T)\otimes H(BT))^W.$$ When $X = X_n(\pm \id)$, we have $X^T \cong T^n \cup T^n$ by Lemma \ref{twotoris}. On $X_n(+\id)^T$, $W$ preserves the components and acts diagonally on each copy of $T^n$ via the usual Weyl group action on $T$ so $$H_T(X_n(+\id)^T)^W \cong 2 H_T(T^n/W)$$ from which the formula above follows. On $X_n(-\id)^T$, $W \cong \Z_2$ interchanges the connected components, so $$H_T(X_n(-\id)^T)^W \cong H_T(T^n)$$ from which the formula follows. 

Finally, for $(X_n^r \times G/T)^T = (X_n^{r})^T \times (G/T)^T$ and $W$ transposes $(G/T)^T \cong S^0$, so $$H_G((X_n^r\times G/T)^T)^W \cong H_T((X_n^r)^T)$$  and the formula follows. 
\end{proof}

\begin{prop}\label{jader}
Let $X= X_n(\mathcal{C})$ be the character variety defined in \S \ref{resout} . The Poincar\'e polynomial of the pair $(X/G, X^T/W)$ satisfies:

\begin{multline*}
  P_t(X/G, X^T/W )= \\ \begin{cases}
 t[\dfrac{(1+t)^n}{1-t^2}+\dfrac{(1-t)^n}{1+t^2} - \dfrac{(1+t^3)^n +t^2 (t+t^2)^n}{1-t^4} ] & \text{if $\mathcal{C} = \id$ and n is odd}\\
\\
 t[\dfrac{(1+t)^n}{1-t^2}+\dfrac{(1-t)^n}{1+t^2} - \dfrac{(1+t^3)^n + (t+t^2)^n}{1-t^4} ] & \text{if $\mathcal{C} = \id$ and n is even}\\
\\
t[ \dfrac{(1+t)^n}{1-t^2} - \dfrac{(1+t^3)^n + (t+t^2)^n}{1-t^4} ] & \text{if $\mathcal{C} = -\id$ and n is odd}\\
\\
t[ \dfrac{(1+t)^n}{1-t^2} - \dfrac{(1+t^3)^n + t^2(t+t^2)^n}{1-t^4} ]  & \text{if $\mathcal{C} = -\id$ and n is even}\\
\\
t[  \dfrac{2(1+t)^n}{1-t^2} - \dfrac{(1+t^3)^n + (t+t^2)^n}{1-t^2}]& \text{if $\mathcal{C} \cong G/T$}\\
\end{cases}
\end{multline*}
and $H(X/G, X^T/W)$ has trivial cup product.
\end{prop}

\begin{proof}
Because the spaces $X$ are $G$-equivariantly formal, the Poincar\'e polynomials satisfy:
$$ P_t^G(X) = \frac{P_t(X)}{1-t^4}$$
Simply apply Proposition \ref{helpful}, inputting Poincar\'e polynomials from Lemma \ref{dumbyorkys} and Theorem \ref{onepoly}.
\end{proof}

The Poincar\'e polynomials computed in Proposition \ref{jader} are interesting in their own right. The subspace $X^{irr} := X - G\cdot X^T$ is exactly the set of irreducible representations in $X = Hom_{\mathcal{C}}(\pi_1(\Sigma_n),G)$, so $$H_{cpt} (X^{irr}) \cong H(X/G, X^T/W)$$
where $H_{cpt}(.)$ denotes compactly supported cohomology. The group $G/Z(G) \cong SO(3)$ acts freely on $X^{irr}$ so $X^{irr}/G$ is a manifold whenever $X$ is a manifold.
\begin{cor}\label{hiddencameras}
Let $X = X_n(\mathcal{C})$ for generic $\mathcal{C}$ or let $X = X_n^r$.
The canonical map $\phi \co H_{cpt}(X^{irr}/G) \rightarrow H(X^{irr}/G)$ is zero.
\end{cor}

\begin{proof}
Because $X^{irr}/G$ is an open subset of the compact space $X/G$, we have $H_{cpt}(X^{irr}/G) \cong H(X/G, X^{red}/G)$.
$X^{irr}/G$ is the quotient of a $3n$ dimensional, orientable smooth manifold by the free action by $G/Z(G)$, so is itself an orientable (noncompact) manifold of dimension $3n-3$. Suppose that $\alpha \in H^*_{cpt}(X^{irr})$ satisfies $\phi(\alpha) \neq 0$.  Then by Poincar\'e duality, there exists $\beta \in H^{3n-3-*}(X^{irr})$ such that $\phi(\alpha) \cup \beta = \alpha \cup \beta \in H^{3n-3}_{cpt}(X^{irr})$ is non zero, which contradicts triviality of the cup product.
\end{proof}

Corollary \ref{hiddencameras} is the analogue of Theorem 0.2.1 of Hausel \cite{hau} in the context of moduli spaces of stable $SU(2)$-Higgs bundles, where it was used as evidence in support of a conjecture that the $L^2$-cohomology of the moduli space is trivial.

\begin{thm}\label{5cases}
Let $X = X_n(\mathcal{C})$ be the character variety defined in \S \ref{resout} . 
The Poincar\'e polynomial of the orbit space $X/G$ is

\begin{multline*}
P_t(X/G)= \\ \begin{cases}
 t[\dfrac{(1+t)^n}{1-t^2}+\dfrac{(1-t)^n}{1+t^2} - \dfrac{(1+t^3)^n +t^2 (t+t^2)^n}{1-t^4} - (1+t)^n-(1-t)^n] +1+t & \text{if $\mathcal{C} = \id$ and n is odd}\\
\\
 t[\dfrac{(1+t)^n}{1-t^2}+\dfrac{(1-t)^n}{1+t^2} - \dfrac{(1+t^3)^n + (t+t^2)^n}{1-t^4}  - (1+t)^n - (1-t)^n] +(1+t)(1 + t^n)& \text{if $\mathcal{C} = \id$ and n is even}\\
\\
t[ \dfrac{(1+t)^n}{1-t^2} - \dfrac{(1+t^3)^n + (t+t^2)^n}{1-t^4} - (1+t)^n] +(1+t) (1 + t^n) & \text{if $\mathcal{C} = -\id$ and n is odd}\\
\\
t[ \dfrac{(1+t)^n}{1-t^2} - \dfrac{(1+t^3)^n + t^2(t+t^2)^n}{1-t^4} -(1+t)^n]  +1+t& \text{if $\mathcal{C} = -\id$ and n is even}\\
\\
t[  \dfrac{2(1+t)^n}{1-t^2} - \dfrac{(1+t^3)^n + (t+t^2)^n}{1-t^2} - 2(1+t)^n]  +(1+t)(2 + 2t^n) & \text{if $\mathcal{C} \cong G/T$}.\\
\end{cases}
\end{multline*}

\end{thm}

\begin{proof}

To simplify the discussion, we omit the case $\mathcal{C} = G/T$. We use the long exact cohomology sequence of the pair

$$   \rightarrow H^*(X/G, X^T/W) \rightarrow H^*(X/G) \rightarrow H^*(X^T/W) \rightarrow^{d} H^{*+1}(X/G, X^T/W) \rightarrow .$$
By exactness we obtain the equation

\begin{equation}\label{jadbal}
P_t(X/G) = P_t(X/G,X^T/W) - tP_t(X^T/W) +(1+t)P_t( \ker (\phi))
\end{equation}
and we turn to computing the terms on the right of (\ref{jadbal}). 

It follows easily from the description of the $W$ action on $X^T$ in the proof of Lemma \ref{dumbyorkys} that:

\begin{equation}\label{centrede}
P_t(X^T/W) = \begin{cases}
(1+t)^n  & \text{if $\mathcal{C}= -\id$}\\
(1+t)^n + (1-t)^n & \text{if $\mathcal{C} = \id$}.\\
\end{cases}
\end{equation}
By Lemma \ref{prtolam}, the coboundary map $d$ factors through:

\begin{equation}\begin{CD}\label{monkout}
\xymatrix{ H(X/G) \ar[r]^{i^*} \ar[d]^{q}& H^*(X^T/W) \ar[r]^{d} \ar[d] \ar[dr]^{\psi}& H^{*+1}(X/G,X^T/W) \ar[d]_{p}^{\cong} \ar[r]^{j^*}& H(X/G) \\
             H_T(X)^W \ar[r]^{i^*_T} & H_T^*(X^T)^W \ar[r]^{d_T} & H_T^{*+1}(X, X^T)^W &}
\end{CD}\end{equation}
so $\ker( d) = \ker( \psi)$.
It follow from the description of the localization map in \S \ref{chapter6} that 

\begin{align*}
\ker( \psi)  = 
\begin{cases} H^0(X^T/W)_+ & \text{ if $X = X_n^s$} \\
                          H^0(X^T/W)_+ \oplus H^n(X^T/W)_- & \text{ if $X = X_n^r$} \\
                          \end{cases}
                          \end{align*}
and in particular
\begin{equation}\label{getitfud}
P_t(\ker(d)) = P_t(\ker( \phi))  = 
\begin{cases} 1 & \text{ if $X = X_n^s$} \\
                          1 + t^n & \text{ if $X = X_n^r$}. \\
                          \end{cases}
                          \end{equation}
Substituting Proposition \ref{jader}, (\ref{centrede}) and (\ref{getitfud}) into (\ref{jadbal}) gives the answer.
\end{proof}

\begin{cor}\label{highbark}
For the singular representation variety $X_n^s$, the reduced cohomology of the orbit space $\tilde{H}(X_n^s/G) \cong H^{>0}(X_n^s/G)$ has trivial cup product.
\end{cor}

\begin{proof}
In the singular case, the proof of Theorem \ref{5cases} shows that the natural map $$ H(X_n^s/G, (X_n^s)^T/W) \rightarrow \tilde{H}(X_n^s/ G)$$
is surjective.  By Proposition $\ref{jader}$, $H(X_n^s/G, (X_n^s)^T/W)$ has trivial cup product, so $\tilde{H}(X_n^s/ G)$ must also.
\end{proof}

\subsection{ $X_n(\epsilon)/G \rightarrow SU(2)^n/SU(2)$ is a cohomological orbit space}

Let $X = X_n^r$ or $X_n^s$. The $\Z_2$ action on $X$ defined in Proposition \ref{somemanyprops}, commutes with the $G$ action, so the quotient $X/G$ inherits a $\Z_2$ action. The projection map $\rho \co X \rightarrow G^n$ is $G$-equivariant, and so descends to a map $\overline{\rho} \co X/G \rightarrow G^n/G$. 

\begin{lem}\label{sweatytee}
Let $X$ denote the representation variety $X_n^r$ or $X_n^s$. The map $\overline{\rho} : X/G \rightarrow G^n/G$ is a strong cohomological $\Z_2$-orbit space, inducing an isomorphism $$ H(G^n/G) \cong H(X/G)_+$$ where $H(X/G)_+$ is the +1 eigenspace of the induced $\Z_2$ action.
\end{lem}

\begin{proof}
The first two axioms in Definition \ref{princ} are easily verified. To check axiom 3,  we must show that 
$$ H(\bar{\rho}^{-1}([y])/\Z_2) \cong H(pt) $$ for each $y \in G^n$ representing  $[y] \in G^n/G$. We have

\begin{equation*}
(\overline{\rho})^{-1}([y])/\Z_2 \cong \rho^{-1}(y)/(\Z_2 \times G_y)
\end{equation*}
where $G_y$ is the stabilzer of $y$. By Proposition \ref{somemanyprops}, $\rho^{-1}(y)/\Z_2$ is either a point or $\R P^2$, while $G_y$ must be one of $Z(G)$, $G$ or a maximal torus $T$.

 If $\rho^{-1}(y)/\Z_2$ is a point, then $\rho^{-1}(y)/(\Z_2 \times G_y)$ is also a point. If $\rho^{-1}(y)/\Z_2 \cong \R P^2$ then $\rho^{-1}(y)/(\Z_2 \times G_y)$ is a point, a line segment or $\R P^2$,  depending on $G_y$.  In all cases $h^{-1}(y)$ has trivial cohomology.  
\end{proof}

\subsection{Cup product on orbit space}

We now turn to the cup product structure on the cohomology of the orbit space. In Corollary \ref{highbark}, the product structure on $H(X^s_n/G)$ was shown to be trivial, so we will concentrate on the regular case $X_n^r/G$. 

Considering the diagram (\ref{monkout}), we obtain a short exact sequence of rings

\begin{equation*}
0 \rightarrow \im (j^*) \rightarrow H(X_n^r/G) \rightarrow \ker (d) \rightarrow 0.
\end{equation*}
By (\ref{getitfud}), $R := \ker (d)$ is isomorphic to the truncated polynomial ring $ \Q [\alpha]/(\alpha^2)= Q \{ 1, \alpha\}$, where $\alpha$ lives in degree $n$. 

\begin{lem}\label{strtrek}
The ideal $\im (j^*)$ has trivial cup product. Equivalently, 
\begin{equation}\label{wierd pub}
0 \rightarrow im(j^*) \rightarrow H(X_n^r/G) \rightarrow R \rightarrow 0
\end{equation}
is a Hochschild extension.
\end{lem}

\begin{proof}

We introduce 
\begin{equation}\label{phigh}
\phi := j^* \circ p^{-1} \circ d_T.
\end{equation}
Because $d_T$ is surjective and $p$ is an isomorphism we know that $\im (j^*) = \im ( \phi)$. The image of $d_T$ has trivial cup product, so $\im (\phi) = \im (j^*)$ must as well. 
\end{proof}

Because $\im (j^*)$ has trivial cup product, the cup product action of $H(X/G)$ on $\im(j^*)$ descends to a (bi)module action by $H(X/G)/ \im(j^*) \cong R$. This can be described explicitly:

\begin{prop}\label{sogassy}
The structure of $\im(j^*)$ as a $R$-module satisfies and is determined by:
\begin{equation*}
\alpha \cdot \phi(\beta) = \phi( (\alpha \otimes 1) \cup \beta)
\end{equation*}
where $\alpha$ is the generator of $R \subset H(X^T/W) = H(X^T)^W$ and $\beta \in H_T(X^T)^W = (H(X^T) \otimes H(BT))^W$.
\end{prop}

\begin{proof}
Because the diagram (\ref{monkout}) is induced by a map between pairs of spaces, it is in fact a commutative diagram of $H(X/G)$-modules in the usual way.


Let $\alpha' \in H(X/G)$ satisfy $i^*(\alpha') = \alpha$. Then for $\phi (\beta) \in \im (\phi) = \im(j^*)$, we have 
\begin{align*}\alpha \cdot \phi (\beta) = \alpha' \cup \phi(\beta) = \phi( i_T^*(q(\alpha')) \cup \beta)\\ = \phi( (\alpha \otimes 1) \cup \beta)
\end{align*} as stated.
\end{proof}


\begin{rmk}
Proposition \ref{sogassy} and Lemma \ref{strtrek} determine the cup product structure on $H(X/G)$ almost completely. Given a vector space splitting $s\co R \rightarrow H(X/G)$ of the short exact sequence (\ref{wierd pub}) satisfying $s(1) = 1$, the only remaining uncertainty is the value of the product $s(\alpha) \cup s(\alpha)$. Unfortunately, we are unable to determine this. If $H(X/G)$ is to inherit a bigrading from the bigrading on $H(X^T) \otimes H(BT)$, we can show that the Hochschild class of (\ref{wierd pub}) must vanish, so we suspect that it does.
\end{rmk}

\bibliographystyle{plain}
\bibliography{thesis}

\end{document}